\documentclass{tran-l}
\usepackage{natbib, amssymb,latexsym, amscd}
\usepackage[all]{xy}
\usepackage{graphicx}

 \newtheorem{theorem}{Theorem}[section]
 \newtheorem{cor}[theorem]{Corollary}
 \newtheorem{lemma}[theorem]{Lemma}
 \newtheorem{proposition}[theorem]{Proposition} \theoremstyle{definition}
 \newtheorem{definition}[theorem]{Definition}
 \theoremstyle{definition}
 \newtheorem{example}[theorem]{Example}
 \theoremstyle{remark}
 \newtheorem{rem}[theorem]{Remark}
 \numberwithin{equation}{section}

\newcommand{\ben}{\begin{equation}}
\newcommand{\een}{\end{equation}}

\newcommand{\integer}{\ensuremath{{\mathbb Z}}}
\newcommand{\field}{\ensuremath{{\mathbb F}}}

\newcommand{\complex}{\ensuremath{{\mathbb C}}}

\newcommand{\rational}{\ensuremath{{\mathbb Q}}}

\newcommand{\FF}{{\mathcal F}}

\newcommand{\CC}{\mathcal{C}}

\newcommand{\To}{\longrightarrow}

\newcommand{\tK}[1]{\ensuremath{{}^{#1} K}}
\begin{document}

\title{Stringy product on twisted orbifold K-theory for abelian quotients}

\author{Edward Becerra and Bernardo Uribe}
\thanks{Both authors acknowledge the support of COLCIENCIAS through the grant 120440520246 and of
CONACYT-COLCIENCIAS throught contract number 376-2007. The second
author was partially supported by the ``Fondo de apoyo a
investigadores jovenes" from Universidad de los Andes. }

\address{Departamento de Matem\'{a}ticas, Universidad de los Andes,
Carrera 1 N. 18A - 10, Bogot\'a, COLOMBIA}
 \email{ es.becerra75@uniandes.edu.co \\
buribe@uniandes.edu.co }
 \subjclass[2000]{Primary 14N35, 19L47 ; Secondary 55N15, 55N91}
\keywords{Stringy product, twisted orbifold K-theory, Chen-Ruan cohomology, inverse transgression map}
\begin{abstract}
In this paper we present a model to calculate the stringy product
on twisted orbifold K-theory of Adem-Ruan-Zhang for abelian complex orbifolds.

 In the first part we consider the non-twisted case on  an orbifold  presented as the quotient of a manifold acted by a
compact abelian Lie group. We give an explicit description of the obstruction bundle, we explain the relation with the
product defined by Jarvis-Kaufmann-Kimura and, via a Chern
character map, with the Chen-Ruan cohomology, and we explicitely
calculate the stringy product for a weighted
projective orbifold.

In the second part we consider orbifolds presented as the quotient of a manifold acted by a
finite abelian group and twistings coming from the group cohomology. We show a decomposition
 formula for twisted orbifold K-theory that is suited to calculate the stringy product and we use this formula
 to calculate two examples when the group is $(\integer/2)^3$.
\end{abstract}
\maketitle

% \tableofcontents

\section{Introduction}

This paper is devoted to the study of the stringy product on the twisted orbifold K-theory
 defined by Adem-Ruan-Zhang
\cite{AdemRuanZhang}.

In the first part we consider the non-twisted case on  an orbifold  presented as the quotient of a manifold acted by a
compact abelian Lie group. By means of methods used to understand the
obstruction bundle developed in \cite{ChenHu,
JarvisKaufmannKimura}, we were able to give a very explicit
description of the obstruction bundle in terms of $G$-equivariant
bundles. With this model in hand, the calculation of the
obstruction bundle becomes a simple diagonalization procedure and
a projection onto some of the components. One of the key
advantages of this construction is that everything can be
performed without the use of rational coefficients, and therefore
the torsion classes of equivariant K-theory are not disregarded.

We construct a stringy ring on the K-theory of the Borel
construction of the inertia orbifold, similar to the stringy
structure in K-theory that was defined in
\cite{JarvisKaufmannKimura}. This construction does not need
rational coefficients either, and uses the same principle of the
stringy product.

Applying a calibrated Chern character map (mainly due to
Jarvis-Kaufmann-Kimura \cite{JarvisKaufmannKimura}), we were able
to prove that there is a ring homomorphism between the stringy
ring on the K-theory of the inertia orbifold and the Chen-Ruan
cohomology of the orbifold. We come to the conclusion that the
right place to study stringy products is K-theory and not
cohomology, this due to the fact that the torsion classes in the
stringy K-theory are present.

We calculate an explicit example of the rings that we have defined
in the case of a weighted projective orbifold, and for doing so,
we make use of a decomposition of equivariant K-theory developed in
\cite{AdemRuan}

In the second part we consider orbifolds presented as the quotient of a manifold acted by a
finite abelian group and twistings coming from the group cohomology. We recall the definition
of the twisting classes via the inverse transgression map and we give a simple argument to show that these classes must
be torsion. We calculate the inverse transgression map in the case that the group is $(\integer/2)^3$ and we expand on
an example previously constructed in \cite{AdemRuanZhang}.

Then we show a decomposition formula for twisted orbifold K-theory
that can be considered as an hybrid between the decomposition
formula for equivariant K-theory using cyclotomic fields and the
decomposition for twisted orbifold K-theory (both in
\cite{AdemRuan}). We finish the paper using this formula to calculate  the stringy product in twisted orbifold
K-theory with rational coefficients when the orbifolds are $[*/(\integer/2)^3]$ and $[T^6/
(\integer/2)^3]$, and the twisting classes are non-trivial.

 It was pointed out to us by A. Adem that the decomposition formula
is not necessary to calculate the stringy product in twisted orbifold
K-theory in the case of $[*/(\integer/2)^3]$ and an explicit calculation of this ring was carried out by A. N. Duman in \cite{Duman}.

\subsection{Acknowledgements}
We would like to thank conversations with professors Alejandro
Adem, Shengda Hu, Ernesto Lupercio and Yongbin Ruan. The second
author would like to thank the hospitality of the Pacific
Institute for the Mathematical Sciences PIMS were parts of this
work were carried out. Both authors acknowledge the support of COLCIENCIAS through the grant number 120440520246.

\section{Stringy product on the K-theory of the inertia orbifold} \label{section stringy product in orbifold K-theory}

Let  $M$ be a compact manifold and $G$ a compact abelian Lie group
that acts almost freely on $M$, such that the orbifold $[M/G]$ is
endowed with an almost complex structure. In what follows we will
define a stringy product on the equivariant K-theory of the
inertia orbifold of $[M/G]$ in as much the same way as the
Chen-Ruan product in orbifold cohomology. We will show that our
definition will be equivalent to the one defined in
\cite{AdemRuanZhang} without twisting. The key ingredient of the
definition will be the obstruction bundle and its construction
will be outlined in this section.

Consider the set $\CC =\{ g \in G |   M^g \neq \phi \}$ where
$M^g$ denotes the fixed point set of the action of $g$. As $G$ is
compact abelian and the action is almost free, the set $\CC$ is
finite.

In the case that we are interested, the inertia orbifold (see
\cite{Moerdijk2002}) of $[M/G]$ is
$$I[M/G] = \sqcup_{g \in \CC} [M^g  /G] $$
where $G$ acts in the natural way.

 \begin{definition}
The stringy K-theory of the inertia orbifold $K_{st}([M/G])$, as a
$\integer$-module, is defined as the K-theory of its inertia
orbifold, i.e.
$$K_{st}^*([M/G]) := K^*(I[M/G]) \cong \bigoplus_{g \in \CC} K^*_G(M^g) \times\{g\}.$$
Note that the summands in the right hand side are in
$G$-equivariant K-theory \cite{Segal_K-theory}.
\end{definition}

The ring structure of the stringy K-theory defined in
\cite{AdemRuanZhang} uses the genus zero, three point function,
Gromov-Witten invariants approach. In this paper we will show an
equivalent description for this ring structure that has the
advantage that it is very simple to use when one wants to
calculate explicit examples.

\subsection{Obstruction bundle} \label{subsection obstruction bundle}

 Let us start by noting that as the orbifold $[M/G]$ is
almost complex, then the normal bundles of the fixed point sets
$M^g$ are all $G$-equivariant complex bundles. Now, take $g,h \in
\CC$ and consider the space $M^{g,h} := M^g \cap M^h$. As $g$ and
$h$ commute, then one can simultaneously diagonalize the action of
$g$ and $h$ on the normal bundle $V_{g,h}$ of the inclusion
$M^{g,h} \to M$.
Therefore, as the order of both $g$ and $h$ is finite,
 $V_{g,h}$ is a  complex bundle isomorphic to
$V_1 \oplus \cdots \oplus V_k$ where $V_j$ is simultaneously the
eigen bundle of $e^{2\pi i a_j}$ for the action of $g$ and of $e^{2\pi i b_j}$
for the action of $h$, with $a_j, b_j \in \rational$  and $0 \leq a_j, b_j <1$. Note
that (and this is a key fact) as the group $G$ is abelian, then
all the bundles $V_j$ are $G$-equivariant (this construction is
performed in each connected component of $M^{g,h}$).

\begin{definition} \label{Obstruction bundle}The obstruction bundle $D_{g,h}$ for the pair
$g,h$ is defined as:
$$D_{g,h} := \bigoplus_{\{j \ | \ a_j + b_j >1\}} V_j,$$
and is a $G$-equivariant bundle over $M^{g,h}$.
\end{definition}

\begin{rem} \label{remark ChenHu} If we take $k=(gh)^{-1}$ and $e^{2\pi i c_j}$ the eigenvalue of the
action of $k$ in $V_j$ ($0 \leq c_j <1$), then we have also that
$$D_{g,h} = \bigoplus_{\{j \ | \ a_j + b_j + c_j=2\}} V_j.$$ This bundle is
the obstruction bundle that is given in \cite[Prop. 1 page
65]{ChenHu} for abelian orbifolds. We are adding the fact that
this bundle is $G$-equivariant.
\end{rem}

The obstruction bundle $D_{g,h}$ defined above is equivalent to
the obstruction bundle given in \cite[Def.
1.5]{JarvisKaufmannKimura} in the case that the group $G$ is
finite. Let's recall their definition.

For $g \in \CC$ split the normal bundle $W_g$ of the inclusion
$M^g \to M$ into eigen bundles of the $g$ action. Let $W_g^l$ be
the eigen bundle of $e^{2 \pi i a_l}$ and define the bundle (in $K^0(M^g)
\otimes \rational$)
\begin{eqnarray} \label{bundles Fg}
 \FF_g := \bigoplus_l a_j W_g^l.
\end{eqnarray}
 The obstruction bundle
$B_{g,h}$ of \cite{JarvisKaufmannKimura} is
\begin{eqnarray} \label{bundle Bgh}
B_{g,h} := TM^{g,h} \ominus TM|_{M^{g,h}} \oplus \FF_g|_{M^{g,h}}
\oplus \FF_h|_{M^{g,h}}\oplus
\FF_{(gh)^{-1}}|_{M^{g,h}}.\end{eqnarray}

\begin{lemma} When the group $G$ is finite, the bundle  $D_{g,h}$ and the virtual bundle $B_{g,h}$ represent the same
class in $K^0(M^{g,h}) \otimes \rational$.
\end{lemma}

\begin{proof}
Let's write the virtual bundle $B_{g,h}$ in terms of the bundles $V_j$. By
restriction we have that $$\FF_g|_{M^{g,h}} \cong \bigoplus_j a_j
V_j, \ \ \ \ \ \FF_h|_{M^{g,h}} \cong \bigoplus_j b_j V_j \ \
\mbox{and} \  \ \FF_{(gh)^{-1}}|_{M^{g,h}} \cong \bigoplus_j c_j
V_j.
$$ Then
$$B_{g,h} \cong \bigoplus_j a_j V_j \oplus \bigoplus_j b_j V_j
\oplus \bigoplus_j c_j V_j \ominus \bigoplus_j V_j \cong
\bigoplus_j (a_j + b_j + c_j -1)V_j = D_{g,h}.
$$
\end{proof}

\subsection{Stringy Product on K-theory of the inertia orbifold}

With the obstruction bundle in hand, we can define now the stringy
product on the K-theory of the inertia orbifold. Let's denote the
inclusions by
$$e_1 : M^{g,h} \to M^g \ \ \ \ \ e_2: M^{g,h} \to M^h \ \ \ \ \ \ e_3: M^{g,h} \to M^{gh}.$$

\begin{definition} \label{stringy product}
For the almost complex orbifold $[M/G]$,  define the stringy product  $\star$ in the stringy K-theory $K_{st}^*([M/G])$
as follows:
\begin{eqnarray*}
\star : (K^*_G(M^g)\times\{g\}) \times (K^*_G(M^h) \times \{h\}) & \to & K^*_G(M^{gh})\times\{gh\} \\
((F,g),(H,h)) & \mapsto & (e_{3!}(e_1^*F \otimes e_2^*H \otimes
\lambda_{-1}(D_{g,h})),gh),
\end{eqnarray*}
where $g, h \in \CC$, the pullbacks and pushforwards are defined in the equivariant setup and
$\lambda_{-1}(D_{g,h})$ is the Euler class of the bundle $D_{g,h}$, i.e. the restriction to $M^{g,h}$ of the Thom
class of $D_{g,h}$ (see \cite[Prop. 3.2]{Segal_K-theory}).
\end{definition}

\begin{proposition}
The ring structure in $K_{st}^*([M/G])$ defined above in
definition \ref{stringy product} is the same as the ring structure
defined in Definition 7.3 of \cite{AdemRuanZhang}.
\end{proposition}

\begin{proof}
By remark \ref{remark ChenHu} above we know that in Proposition 1
of \cite{ChenHu} it is proved that the obstruction bundle of the
Chen-Ruan cohomology is isomorphic to the obstruction bundle
$D_{g,h}$. As in \cite{AdemRuanZhang} the obstruction bundle is
the obstruction bundle of the Chen-Ruan cohomology
\cite{ChenRuan}, then the two ring structures agree.
\end{proof}

\subsection{Associativity} \label{section associativity}

The fact that $K_{st}^*([M/G])$ together with the stringy product in
K-theory is associative is proved in Theorem 7.4
\cite{AdemRuanZhang}.
 It is basically the same proof that is given in lemma 5.4 of \cite{JarvisKaufmannKimura} but in the equivariant setup.
 We are not going to reproduce the proof in here because it would be redundant. Instead, we will write some words on the
 key ingredient of the proof which is the equivariant version of the excess intersection formula of Quillen.

 In Proposition 3.3 of \cite{Quillen} it is shown that for the inclusion of almost complex manifolds
 $$\xymatrix{
 S_1 \ar[r]^{i_1} & S \\
 V \ar[r]^{j_2} \ar[u]^{j_1} & S_2 \ar[u]_{i_2}
 }$$
 with $V=S_1 \cap S_2$, then $$i_2^* i_{1*} x = j_{2*}( e(E(S, S_1, S_2)) \cdot j_1^*x)$$ holds in cohomology where
 $e(E(S, S_1, S_2)$ is the Euler class of the bundle
 $$ E(S, S_1, S_2) = TS|_V \oplus TV \ominus TS_1|_V \ominus TS_2|_V.$$
 This bundle is known as the excess intersection bundle.

  The formula also holds in K-theory, namely
 $$i_2^* i_{1!} F = j_{2!}( \lambda_{-1} (E(S, S_1, S_2)) \otimes j_1^*F).$$
Both formulas rely on the fact that for a complex bundle $E \to V$
with zero section $\sigma: V \to E$, one has $\sigma^* \sigma_* x
= e(E) \cdot x$ in cohomology and $\sigma^* \sigma_! F =
\lambda_{-1}(E) \otimes F$ in K-theory.

In the case that the bundle $E \to V$ is equivariant, Segal
 showed that the formula $\sigma^* \sigma_! F
= \lambda_{-1}(E) \otimes F$ also applies to the equivariant setup
\cite[Page 140]{Segal_K-theory}; therefore the excess intersection
formula of Quillen holds in the case that the manifolds $S, S_1,
S_2,V$ and the maps $i_i, j_1, i_2, j_2$  are all equivariant.

With the equivariant excess intersection formula in hand, the
associativity boils down to check the following equality in
$K_G^0(N)$ where $N=M^{g,h,k}$:
$$D_{g,h}|_N \oplus D_{gh,k}|_N \oplus E(M^{gh},M^{g,h}, M^{gh,k}) = D_{g,hk}|_N \oplus
D_{h,k}|_N \oplus E(M^{hk}, M^{g,hk}, M^{h,k}).$$ The proof
follows the same lines as the proof of lemma 5.4 of
\cite{JarvisKaufmannKimura}; we only need to remark that in our
case all the operations are done in the $G$-equivariant setup. We
will not reproduce the proof.

\subsection{Weighted projective orbifold $\complex P[p:1: \dots :1]$}
\label{section example} In this section we are going to calculate
the stringy K-theory in the case of the weighted projective
orbifold $\complex P[p:1: \dots :1]$. Let's recall its definition.

   Let $M=S^{2n+1} = \{
(z_0, \dots, z_n) \in \complex^{n+1} | \sum_i |z_i|^2 =1 \}$ and
$G =S^1$ acting on $M$ as
\begin{eqnarray}
\nonumber S^1 \times S^{2n+1} & \to & S^{2n+1} \\
(\lambda, (z_0, \dots, z_n)) & \mapsto & (\lambda^p z_0, \lambda
z_1, \lambda z_2, \dots, \lambda z_n). \label{action of the group}
\end{eqnarray}
with $p$ a prime number. The quotient space $M/G$ is the weighted
projective space $\complex P (p, 1, \dots, 1)$ and the orbifold
 $[M/G]$ is the weighted projective orbifold $ \complex P[p:1: \dots :1]$.

The set $\CC$ is equal to the elements in the group $\integer/p
\subset G$ which we will denote by
 $\{ g_0=id, g_1, \dots, g_{p-1} \}=\CC$. The fixed point sets are $M^{g_0} =M$
and $M^{g_j} = \{ (z_0, 0, \dots, 0) \in M\} = S^1$ for $1\leq j <
p$. The $G$ action on $M^{g_j}$ is the same as the action of $G$
on the space $S^1/(\integer/p)$ given by rotation. Then as a
$\integer$-module
$$K_{st}^*([M/G]) = K_{G}^*(S^{2n+1})\times \{g_0\} \oplus \bigoplus_{j=1}^{p-1} K_{G}^*(S^1/(\integer/p)) \times \{g_j\}.$$

As the action of $G$ in $ M^{g_j} \cong S^1/(\integer/p)$ is
by rotation, then we know that
$$K_G^*(S^1/(\integer /p)) \cong K_{\integer /p}^* (pt) \cong
R(\integer /p),$$ where $R(\integer /p)$ is the Grothendieck ring
of $\integer /p$ representations. Therefore, the stringy K-theory,
as a $\integer$-module is isomorphic to

$$K_{st}^*([M/G]) = K_{G}^*(S^{2n+1})\times \{g_0\} \oplus \bigoplus_{j=1}^{p-1} R(\integer/p) \times \{g_j\}.$$

For $g_j, g_k \in \CC$, $j \neq 0$,  the obstruction bundle
$D_{j,k} :=D_{g_j,g_k}$ is a subbundle of the normal bundle
$V_{j,k}$ of the inclusion $S^1 \to S^{2n+1}$, where in this case
$M^{g_j,g_k} = M^{g_j}=S^1$. The normal bundle $V_{j,k}$ is isomorphic to the
trivial bundle $S^1 \times \complex^n$ over $S^1$, where the
circle action is given by $\lambda (\eta, {\bf z}) = (\lambda^p
\eta, \lambda {\bf z})$. This follows from the fact that one can take the set
$$\{(z_0, \dots, z_n ) \in S^{2n+1} \ | \ |z_i| < \frac{1}{n} \ \mbox{for} \ 1 \leq i \leq n \}$$ as tubular neighborhood
of the inclusion  $S^1 \to S^{2n+1}$, and this set is isomorphic to
$$S^1 \times \{(z_1, \dots, z_n) \in \complex^n \ | \ |z_i|< \frac{1}{n} \}.$$
 Then the normal bundle can be seen as $nW
\in R(\integer/p) \cong K^*_{S^1}(M^{g_j,g_k})$ where $W$ is the
one dimensional irreducible representation whose character is
$e^{\frac{2 \pi i}{p}}$.

Going back to the obstruction bundle, we have that
\begin{eqnarray} \label{obstruction bundle of weighted} D_{j,k} = \left\{ \begin{array}{ccc}
V_{j,k} & \mbox{for} & j + k > p \\
0 & \mbox{for} & j+k \leq p;
\end{array}  \right.\end{eqnarray}
this is because the action of $g_j$ on the normal bundle $V_{j,k}$
is given by multiplying by $e^{\frac{2 \pi i}{p}j}$.

Having defined the obstruction bundle, we need to explicitly
determine the pullbacks and pushforwards of the inclusion maps.
The first step is to calculate explicitly $K_G^*(M)$, and for this
we are going to use a decomposition theorem for orbifold K-theory
due to Adem and Ruan \cite{AdemRuan} (cf. \cite{LueckOliver})

\begin{theorem}\cite[Th. 5.1]{AdemRuan} \label{theorem AdemRuan}
Let $[M/G]$ be a compact orbifold where $G$ is a compact Lie group
acting almost freely on $M$. Then there is a rational isomorphism
of rings
$$K_G^*(M) \otimes \rational \cong \prod_{\{(C) \ | \ C \subset G \
cyclic\}} \left[K^*(M^C/Z_G(C)) \otimes \rational (\zeta_{|C|})
\right]^{W_G(C)}$$ where $(C)$ ranges over conjugacy classes of
cyclic subgroups, $W_G(C)=N_G(C)/Z_G(C)$ and $\rational
(\xi_{|C|})$ is the cyclotomic field where $\zeta_{|C|}$ is a
primitive $|C|-th$ root of unity.
\end{theorem}

For a cyclic group $C$, the map from left to right is the
composition of maps $$K_G^*(M) \stackrel{\phi}{\To}
K_{N_G(C)}^*(M^C) \stackrel{\psi}{\To}  K_{Z_G(C)}^*(M^C) \otimes
R(C) \To \ \ \ \ \ \ \ \ \ \ \ \ \ \ \ $$
$$ \ \ \ \ \ \ \ \ \stackrel{\alpha}{\To} K^*(EG \times_{Z_G(C)}  M^C) \otimes R(C)
\stackrel{\beta}{\To} K^*(M^C/Z_G(C)) \otimes \rational
(\zeta_{|C|})$$ where: $\phi$ is the restriction map; $\psi$ is
defined such that for any $N_G(C)$ vector bundle $E \to M^C$
$$\psi([E]) = \sum_{V \in Irr(C) } [Hom_C(V,E)] \otimes [V];$$
$\alpha$ is the natural map from equivariant K-theory to the
K-theory of the Borel construction; and $\beta$ is the map
obtained from the isomorphism
$$K^*(EG \times_{Z_G(C)} M^C) \otimes \rational \cong
K^*(M^C/Z_{G(C)}) \otimes \rational$$ that comes from the rational
acyclicity of the fibers of the map $EG \times_{Z_G(C)} M^C \to
M^C/Z_{G(C)}$, and the ring map $R(C) \otimes \rational \to
\rational(\zeta_{|C|})$ whose kernel is the ideal of elements
whose characters vanish in all generators of $C$.

\begin{rem}
We have modified slightly the statement of Theorem 5.1 in \cite{AdemRuan} by leaving the decomposition
formula only in terms of K-theory. Applying the Chern character isomorphism  $$Ch: K^*(M^C/Z_{G(C)}) \otimes \rational \stackrel{\cong}{\To}
H^*(M^C/Z_{G(C)})\otimes \rational$$ one recovers the original statement.
\end{rem}

Going back to the weighted projective orbifold $M/G$, we have

\begin{lemma} \label{equiv of M}
Let $G=S^1$ act on $M = S^{2n+1}$ as in (\ref{action of the
group}), then there is an isomorphism of rings
$$K^*_G(M) \otimes \rational \cong \rational[u]/\langle
(1-u)^{n+1} \rangle \times \rational (\zeta)$$ where
$\zeta=e^{\frac{2 \pi i}{p}}$.
\end{lemma}
\begin{proof}
From theorem \ref{theorem AdemRuan} we have that
$$K_G^*(M)\otimes \rational \cong \left(K^*(M/G) \otimes \rational \right) \times \left( K^*(M^{\integer/p}/G) \otimes \rational(\zeta)
\right)$$ because the only cyclic subgroups of the circle that have non empty fixed point set are
the trivial subgroup and $\integer/p$.  The quotient spaces are: $M/G = \complex P(p:1:
\dots :1)$ and $M^{\integer/p}/G = point$, therefore
$$K_G^*(M)\otimes \rational \cong \left(K^*(\complex P(p:1: \dots :1)) \otimes \rational \right) \times
\rational(\zeta).$$ We are left to calculate the K-theory  of the
weighted projective space $\complex P^n_w:=\complex P(p:1: \dots
:1)$; this we will do by calculating its cohomology.

The space $\complex P^n_w$ has only one orbifold point whose
neighborhood looks like the quotient of $\complex^n$ by the
diagonal action of $\integer /p$. Then the space $\complex P^n_w$
can be thought as a smooth manifold because the orbifold point can
be removed.

If we consider the inclusion $j: \complex P^{n-1} \to \complex
P^n_w$ that skips the first variable, then we get the following
short exact sequence:
$$0 \To H^{2n}(\complex P^n_w, \complex P^{n-1}) \To H^*( \complex
P^n_w) \stackrel{j^*}{\To} H^*(\complex P^{n-1}) \To 0.$$

As $H^*(\complex P^{n-1})= \integer[y]/ \langle y^n \rangle$ then
there must exist $x \in H^2(\complex P^n_w)$ with $j^*x = y$. As
the transversal intersection of the submanifold $j(\complex
P^{n-1})$ with itself $n-1$ times  inside $\complex P^n_w$ gives
$p$ points (this is because the normal bundle of the inclusion $j$
is isomorphic to the $p$-th power of the canonical line bundle
over $\complex P^{n-1}$), then $x^n$ must be $p$-times the image
of the generator of $H^{2n}(\complex P^n_w, \complex P^{n-1})$.

Therefore we can conclude that
$$H^*( \complex P^n_w) \cong \integer[x,v]/ \langle x^n-pv,
x^{n+1}, v^2 \rangle,$$ with rational coefficients and taking $z:=
p^{-\frac{1}{n}} x$
$$H^*( \complex P^n_w, \rational) \cong \rational[z]/ \langle z^{n+1}
\rangle,$$ and by the Chern character isomorphism
$$K^*(\complex P^n_w) \otimes \rational \cong \rational[u]/ \langle (1-u)^{n+1}
\rangle$$ where $ch(1-u)^n=z^n$.
\end{proof}

For the fixed point set $M^{\integer /p}$ we already know that
$K_G(M^{\integer /p}) \cong R(\integer /p) \cong \integer [W] /
\langle W^p -1 \rangle$. Applying again theorem \ref{theorem
AdemRuan} we have that
\begin{lemma} \label{lemma fixed point equiv}
There is a ring isomorphism
\begin{eqnarray*}
R(\integer /p) \otimes \rational & \cong & \rational \otimes
\rational(\zeta) \\
f(W) & \mapsto & ( f(1),f(\zeta))
\end{eqnarray*}
where $f(W)$ is a polynomial that belongs to $\integer [W] /
\langle W^p -1 \rangle$ and $W$ is the one dimensional irreducible
representation of $\integer /p$ whose character on $ 1 \in
\integer /p$ is $e^{\frac{2\pi i }{p}}$.
\end{lemma}

\begin{lemma}
The pullback of the inclusion map $i : M^{\integer /p} \to M$
$$i^* : K_G^*(M) \otimes \rational  \to  K_G^*(M^{\integer /p})
\otimes \rational$$ is isomorphic to
\begin{eqnarray}
 \label{pullback}
\rational[u]/\langle (1-u)^{n+1}\rangle \times \rational(\zeta) & \to & \rational \times \rational(\zeta) \\
(p(u), q(\zeta)) & \mapsto & (p(1),q(\zeta)). \nonumber
\end{eqnarray}
\end{lemma}

\begin{proof}
It follows directly from lemmas \ref{lemma fixed point equiv} and
\ref{equiv of M}.
\end{proof}
 The pushforward $i_!$ of the inclusion follows the same lines.
 \begin{lemma}
The pushforward
\begin{eqnarray*}
i_! : K_G^*(M^{\integer /p}) \otimes \rational &\to& K_G^*(M) \otimes
\rational
\end{eqnarray*}
is isomorphic to
\begin{eqnarray*}
i_! : \rational \times \rational(\zeta)  &\to&  \rational[u]/\langle (1-u)^{n+1}\rangle \times \rational(\zeta)  \\
(r, q(\zeta)) & \mapsto & (r(1-u)^n, q(\zeta)(1-\zeta)^n)
\end{eqnarray*}
\end{lemma}
\begin{proof}
The application in the first coordinate is  the pushforward in
K-theory of the inclusion of a point into the weighted projective
space $\complex P^n_w$, and the application of the second
coordinate is the image in $\rational (\zeta)$ of the map $i^*i_!
: R(\integer /p ) \to R (\integer /p)$
$$i^* i_!q(W)= q(W) \lambda_{-1}(n W)=q(W)
\lambda_{-1}(W)^n=q(\xi)(1-W)^n$$ where $n W$ represents the class
of the normal bundle of the inclusion $i$ as a $G$ equivariant
bundle.
\end{proof}

We can finish this section by giving the explicit calculation of
the stringy product in the orbifold K-theory of $[M/G]=\complex
P[p: 1: \dots :1]$.

\begin{theorem} \label{theorem stringy Ktheory weighted proj}
The stringy ring with rational coefficients of the K-theory of the
inertia orbifold of the weighted projective orbifold
$[M/G]=\complex P[p: 1: \dots :1]$ is
$$K_{st}^*([M/G])\otimes \rational \cong \left(\rational[u]/\langle (u-1)^{n+1}\rangle \times \rational(\zeta)
\right) \times \{g_0\} \oplus \bigoplus_{j=1}^{p-1} \left(
\rational \times \rational(\zeta) \right) \times \{g_j\}$$ with
the product defined by the following cases:
\begin{itemize}
\item ($j=0$, $k\neq 0$)  $$(p(u),q(\zeta),g_0) \star
(r,w(\zeta)),g_k) = (p(1)r, q(\zeta)w(\zeta),g_{k})$$ \item ($j+k
<p$, $j \neq 0$, $k \neq 0$)
$$(r,q(\zeta),g_j) \star (s,w(\zeta)),g_k) = (rs,
q(\zeta)w(\zeta),g_{j+k})$$ \item ($j+k=p$, $j \neq 0$, $k \neq
0$)
$$(r,q(\zeta),g_j) \star (s,w(\zeta)),g_k) = (rs(1-u)^n, q(\zeta)w(\zeta)(1-\zeta)^n,g_0)$$
\item ($j+k>p$)
$$(r,q(\zeta),g_j) \star (s,w(\zeta)),g_k) = (0, q(\zeta)w(\zeta)(1-\zeta)^n,g_{j+k-p})$$
\end{itemize}

\end{theorem}

\begin{proof}
From formula \ref{obstruction bundle of weighted} we know that the
if obstruction bundle is different than zero then it is precisely
the normal bundle. Therefore the Euler class of the obstruction
bundle is $1$ in the case that the bundle is zero, and $(1-
\zeta)^n$ in the case that it is the normal bundle. The product
structure $\star$ follows from the definition \ref{stringy
product} and of lemmas \ref{equiv of M} and \ref{lemma fixed point
equiv}.
\end{proof}

\subsection{Relation with Chen-Ruan Cohomology} \label{section
relation with CR}

In what follows we want to prove that there is a homomorphism of
rings from the stringy ring of the K-theory of the inertia
orbifold to the Chen-Ruan cohomology. We will show this fact by
first constructing a ring structure on the K-theory of the Borel
construction of the inertia orbifold. Then, via the some modified
version of the Chern character map, we will construct a ring
structure on the cohomology with rational coefficients of the
Borel construction of the inertia orbifold. This latter ring, once
tensored with $\rational$,
  will be isomorphic to the Chen-Ruan cohomology.

\subsection{Stringy K-theory for the Borel construction of the
inertia orbifold} Let's start by reviewing the relation between
equivariant K-theory and the K-theory of the Borel construction
(see \cite{AtiyahSegal_completion}).

For any $G$-space $N$ with $N_G:=N \times_G EG$, let's denote the
map $\tau_N: K_G(N) \to K(N_G)$ that takes any $G$-equivariant
complex vector bundle $F$ over $N$ and maps it to $\tau_N(F) =
F_G:= F\times_G EG$ a complex vector bundle over $N_G$. This map is a ring homomorphism and has very nice properties:
it behaves well under restriction and pushforwards, namely, if we
have an inclusion of $G$-manifolds $i: W \to N$, then $$
\tau_Wi^*(L)=i^*\tau_N(L) \ \ \ \ \mbox{and} \ \ \ \ \tau_N i_!(F)
= i_! \tau_W(F);$$ and it behaves well under the map
$\lambda_{-1}$, i.e.
$$\tau_N(\lambda_{-1}(L))= \lambda_{-1}(\tau_N(L)).$$
Moreover, the map $\tau_N$ realizes the famous Atiyah-Segal completion theorem \cite{AtiyahSegal_completion}.

\begin{theorem} \label{Theorem stringy product K-theory}
Let the orbifold $[M/G]$ be as in section \ref{section stringy
product in orbifold K-theory}. The K-theory of the Borel
construction of the inertia orbifold
$$K_{orb}^*([M/G]) := \bigoplus_{g \in \CC} K^*(M^g_G) \times\{g\}$$
can be endowed with the ring structure
\begin{eqnarray*}
\bullet : (K^*(M^g_G)\times\{g\}) \times (K^*(M^h_G) \times \{h\}) & \to & K^*(M^{gh}_G)\times\{gh\} \\
((F,g),(H,h)) & \mapsto & (F \bullet H,gh).
\end{eqnarray*}
with $F \bullet H =e_{3!}(e_1^*F \otimes e_2^*H \otimes
\lambda_{-1}(\tau_{M^{g,h}}(D_{g,h}))$.
\end{theorem}
\begin{proof}
The associativity of the product $\bullet$ follows from the
associativity of $\star$ (see section \ref{section associativity})
and the properties of the maps $\tau$.
\end{proof}

\begin{rem}
The construction above of theorem \ref{Theorem stringy product K-theory}
 is due to Jarvis-Kaufmann-Kimura
\cite{JarvisKaufmannKimura} when one considers the coarse moduli space
of the inertia orbifold. Ours generalizes it to the Borel
construction.

\end{rem}

\begin{theorem}
The maps \begin{eqnarray*} \tau_{M^g} : K^*_G(M^g) \to K^*(M^g_G)
\end{eqnarray*}
for $g \in \CC$ define a ring homomorphism $$\tau :
(K^*_{st}([M/G]), \star) \to (K^*_{orb}([M/G]),\bullet).$$
\end{theorem}

\begin{proof}
Take $F \in K^*_G(M^g)$ and $H \in K_G^*(M^h)$ and consider the
following set of equalities: \begin{eqnarray*} \tau_{M^{gh}}(F
\star H) & = & \tau_{M^{gh}}\left(e_{3!}\left(e_1^*F \otimes
e_2^*H \otimes
\lambda_{-1}(D_{g,h})\right)\right) \\
&  = & e_{3!}\left(\tau_{M^{g,h}} \left(e_1^*F \otimes e_2^*H
\otimes \lambda_{-1}(D_{g,h})\right) \right) \\
& = & e_{3!}\left( e_1^*(\tau_{M^g}F) \otimes e_2^*(\tau_{M^h} H)
\otimes \lambda_{-1}(\tau_{M^{g,h}} D_{g,h}) \right) \\
& = & \tau_{M^g}F \bullet \tau_{M^h} H.
\end{eqnarray*}
The theorem follows.
\end{proof}

\subsection{Chern Character}

We can define a stringy product on the cohomology of the Borel
construction of the inertia orbifold of $[M/G]$ in the same way
that it was done in the previous section.

\begin{proposition}
The cohomology of the Borel construction of the inertia orbifold
$$H^*_{orb}([M/G]) := \bigoplus_{g \in \CC} H^*(M^g_G, \integer) \times\{g\}$$
can be endowed with the ring structure
\begin{eqnarray*}
\circ : (H^*(M^g_G, \integer)\times\{g\}) \times (H^*(M^h_G, \integer) \times \{h\}) & \to & H^*(M^{gh}_G, \integer)\times\{gh\} \\
((\alpha,g),(\beta,h)) & \mapsto & (\alpha \circ \beta,gh).
\end{eqnarray*}
with $\alpha \circ \beta  =e_{3*}(e_1^*\alpha \cdot e_2^* \beta
\cdot {\rm{eu}}(\tau_{M^{g,h}}(D_{g,h}))$.
\end{proposition}
\begin{proof}
The associativity of the product $\circ$ follows also from the
associativity of product $\star$ (see section \ref{section
associativity}), the properties of the map $\tau_{M^{g,h}}$ and
the properties of the Euler class ``eu".
\end{proof}

In order to have a well defined ring homomorphism between
$(K^*_{orb}([M/G]), \bullet)$ and $(H^*_{orb}([M/G]),\circ)$ we
need to calibrate the Chern character maps at the level of the
spaces $M^g_G$ with some invertible elements in cohomology. This
construction was developed in \cite{JarvisKaufmannKimura} at the
level of the coarse moduli space of the inertia orbifold. The map
is basically the same at the level of the Borel construction.

To construct this calibrated Chern character we need to recall
some properties of the Thom ismorphisms in K-theory and in
cohomology (see for instance \cite[section 2]{AtiyahSinger3}).

Let $X$ be a manifold and $V$ a complex vector bundle over $X$. Let
$$ \phi : K^*(X) \to K^*(V)$$
$$ \psi : H^*(X, \rational) \to H^*(V, \rational)$$
be the Thom isomorphisms in K-theory and in cohomology
respectively. Then, for any $ u \in K^*(X)$ one has

\begin{eqnarray} \label{pushforwards Ktheory cohomology}
{\rm{ch}}( \phi (u)) = \psi( {\rm{ch}}(u) \cdot
\mu(V))\end{eqnarray}
 where the cohomology class $\mu(V)$ is
defined as
$$\mu(V):= \prod \frac{1-e^{x_i}}{x_i}$$
where the $x_i$'s are the Chern roots of $V$. Moreover, the class
$\mu(V)$ is multiplicative, i.e. $\mu(V \oplus F) = \mu(V) \cdot
\mu(F)$ and measures the difference of the Chern character of the
Euler class in K-theory with the one in cohomology, namely
\begin{eqnarray}
\label{euler classes in K and cohomology} {\rm{ch}}
\lambda_{-1}(V) = {\rm{eu}}(V) \cdot \mu(V).\end{eqnarray}

If we have an inclusion of manifolds $i : X \to Y$ with normal
bundle $V$ and pushforward maps
$$ i_! : K^*(X) \to K^*(Y)$$
$$ i_* : H^*(X, \rational) \to H^*(Y, \rational)$$
then one has the equality
$${\rm{ch}}( i_! u) = i_*( {\rm{ch}}(u) \cdot \mu(V))$$
where in this case the cohomology class $\mu(V)$ has support on
the normal bundle of $X$.

We are now ready to define the Chern character map for the
orbifold K-theory.

\begin{definition} \cite[Formula (55)]{JarvisKaufmannKimura}
The orbifold Chern character is the homomorphism
$${\it{Ch}} : K_{orb}^*([M/G]) \to H_{orb}^*([M/G] ; \rational)$$
such that for all $g \in \CC$ is defined as
\begin{eqnarray*}
{\it{Ch}}: K^*(M^g_G) &\to & H^*(M^g_G ; \rational) \\
F &  \mapsto & {\rm{ch}}(F) \mu^{-1}( \tau_{M^g} \FF_g)
\end{eqnarray*}
where the $G$-bundle $\FF_g$ over $M^g$ is defined in formula
\ref{bundles Fg}.
\end{definition}

The fact that the orbifold Chern character is a ring homomorphism
depends on the following lemma.

\begin{lemma}
In $K_G(M^{g,h})$ the following equality holds $$D_{g,h} \oplus
TM^{gh}|_{M^{g,h}} \ominus TM^{g,h}  = \ominus e_1^*\FF_{g}\ominus
e_2^*\FF_{h} \oplus e_3^*\FF_{gh}.$$
\end{lemma}

\begin{proof}
We know that
$$\FF_{gh} \oplus \FF_{(gh)^{-1}} = TM|_{M^{gh}} \ominus
TM^{gh}.$$ The equality then holds from the fact that $D_{g,h} =
B_{g,h}$ where $B_{g,h}$ is defined in formula \ref{bundle Bgh}.
\end{proof}

Applying the homomorphism $\tau_{M^{g,h}}$ and the map $\mu$ we
get,
\begin{cor} \label{corollary mu}$$\mu(\tau_{M^{g,h}} D_{g,h}) \mu(
\tau_{M^{g,h}}( TM^{gh}|_{M^{g,h}} \ominus TM^{g,h}))
e_3^*\mu^{-1}( \tau_{M^{gh}} \FF_{gh}) = \hspace{3cm}$$ $$
\hspace{8cm} e_1^*\mu^{-1}( \tau_{M^{g}} \FF_{g}) e_2^*\mu^{-1}(
\tau_{M^{h}} \FF_{h}).$$
\end{cor}

\begin{theorem} \cite[Theorem 6.1]{JarvisKaufmannKimura}
The orbifold Chern character is a ring homomorphism
$${\it{Ch}} : (K_{orb}^*([M/G]), \bullet) \to (H_{orb}^*([M/G] ; \rational), \circ).$$
\end{theorem}
\begin{proof}
Take $F \in K^*(M^g_G)$ and $H \in H^*(M^h_G)$ and consider the
following set of equalities:
\begin{eqnarray*}
{\it{Ch}}(F \bullet H) & = & {\rm{ch}}(F \bullet H)
\mu^{-1}(\tau_{M^{gh}} \FF_{gh}) \\
& = &  {\rm{ch}}\left( e_{3!} \left( e_1^*F \otimes e_2^*H \otimes
\lambda_{-1}( \tau_{M^{g,h}} D_{g,h}) \right) \right)
\mu^{-1}(\tau_{M^{gh}} \FF_{gh}) \\
& = & e_{3*} \Big( e_1^* {\rm{ch}} F \ e_2^* {\rm{ch}} H \
{\rm{eu}}(\tau_{M^{g,h}} D_{g,h}) \ \mu( \tau_{M^{g,h}} D_{g,h})
\Big.
 \\
 &  & \hspace{2cm}
  \Big. \mu(\tau_{M^{g,h}}( TM^{gh}|_{M^{g,h}} \ominus TM^{g,h})) \Big) \mu^{-1}(\tau_{M^{gh}} \FF_{gh}) \\
& = & e_{3*} \Big( e_1^* {\rm{ch}} F \ e_2^* {\rm{ch}} H \
{\rm{eu}}(\tau_{M^{g,h}} D_{g,h}) \ \mu( \tau_{M^{g,h}} D_{g,h})
\Big.
 \\
 &  & \hspace{2cm}
  \Big. \mu(\tau_{M^{g,h}}( TM^{gh}|_{M^{g,h}} \ominus TM^{g,h})) \ e_3^*\mu^{-1}(\tau_{M^{gh}} \FF_{gh})\Big) \\
& = & e_{3*} \Big( e_1^* {\rm{ch}} F \ e_2^* {\rm{ch}} H \
{\rm{eu}}(\tau_{M^{g,h}} D_{g,h}) e_1^*\mu^{-1}( \tau_{M^{g}}
\FF_{g}) e_2^*\mu^{-1}(
\tau_{M^{h}} \FF_{h}) \Big)  \\
& = & e_{3*} \Big( e_1^* {\it{Ch}} F \ e_2^* {\it{Ch}} H \
{\rm{eu}}(\tau_{M^{g,h}} D_{g,h})  \Big) \\
& = & {\it{Ch}} F \circ {\it{Ch}} H.
\end{eqnarray*}
From the second to the third line we used formula
\ref{pushforwards Ktheory cohomology}, the fourth line is by the
properties of the pushforward and the fifth line is by corollary
\ref{corollary mu}.
\end{proof}

Using the fact that the Chern character map becomes a $\integer
/2$ graded vector space isomorphism when one tensors the K-theory
with the rational numbers, and because the classes given by $\mu$
are all invertible, we can deduce that

\begin{cor}
The orbifold Chern character is a ring isomorphism when one
tensors the orbifold K-theory with the rationals, i.e.
$${\it{Ch}} : (K_{orb}^*([M/G]) \otimes \rational, \bullet) \stackrel{\cong}{\to} (H_{orb}^*([M/G] ; \rational), \circ).$$
is a ring isomorphism.
\end{cor}

\subsection{Chen-Ruan Cohomology}

The Chen-Ruan cohomology is defined at the level of the quotient
spaces $M^g/G$ where the obstruction class is the Euler class of
the orbibundle $D_{g,h}/G$ over $M^{g,h}/G$ (see \cite{ChenRuan}).
Namely
$$H^*_{CR}([M/G]) := \bigoplus_{g \in \CC} H^*(M^g/G, \rational)$$
with the product structure given by $$\alpha \cdot \beta := e_{3*}
\left( e_1^* \alpha \ e_2^* \beta \ {\rm{eu}}(D_{g,h}/G)
\right).$$

As the projection maps $\pi_g : M^g_G \to M^g/G$ induce
isomorphisms in cohomology $$\pi_g^* : H^*(M^g/G , \rational)
\stackrel{\cong}{\to} H^*(M^g_G; \rational )$$ (because the
cohomology with rational coefficients of the fibers of $\pi_g$ are
acyclic), and they commute with pullbacks and pushforwards, we can
conclude

\begin{theorem}
The maps $\pi_g$ induce an isomorphism of rings
$$\pi^* : H^*_{CR}([M/G]) \to H^*_{orb}([M/G], \rational).$$
\end{theorem}

\begin{proof}
The isomorphism follows from the fact that $$\pi_{g,h}^*
(D_{g,h}/G )\cong \tau_{M^{g,h}} D_{g,h}.$$
\end{proof}

Composing the ring maps $\tau$, ${\it{Ch}}$ and $\pi^*$ we obtain
\begin{cor}
The composition $$\pi^* \circ {\it{Ch}} \circ \tau :
K_{st}^*([M/G]) \to H^*_{CR}([M/G])$$ is a ring homomorphism.
\end{cor}

\begin{example}
{The case of $\complex P [p:1: \dots :1]$}

In  the case of the weighted projective orbifold that we elaborated
in section \ref{section example}, we just need to disregard the
coordinate of the cyclotomic extension of theorem \ref{theorem
stringy Ktheory weighted proj}; then we have that
$$K_{orb}^*([M/G]) \otimes \rational \cong \rational[u,g]/\langle
(1-u)^{n+1}, g^p - (1-u)^n , g^{p+1} \rangle.$$

Applying the orbifold Chern character, which in this case is
simply the Chern character map, we obtain the
 Chen-Ruan orbifold cohomology of the
weighted projective space:
$$H_{orb}^*(\complex P[p:1:\dots :1]) \cong \rational[x,g]/ \langle x^{n+1}, g^p-x^n,g^{p+1} \rangle ,  $$
 where ${\rm{ch}}(1-u)^n=x^n$.
\end{example}

\section{Stringy product on twisted orbifold K-theory}

This chapter is devoted to study some properties of the stringy
product on twisted orbifold K-theory for orbifolds of the type
$[M/G]$ where $G$ is a abelian finite group.

We will first recall the definition of the twisted orbifold
K-theory and its stringy product following \cite{AdemRuanZhang}.
Then we will show that the twistings used for the stringy product,
that come from the inverse transgression, are all torsion. In the case that
the twistings come from the cohomology of the group $G$,   we will
show a decomposition formula that is very suited to study the
stringy product. This decomposition formula can be seen as a
hybrid between the ring decomposition of equivariant K-theory (see
\cite[Thm 5.1]{AdemRuan}) and the group decomposition of twisted
K-theory (see \cite[Thm. 7.4]{AdemRuan}. Finally we will show an
explicit calculation of the stringy product of twisted orbifold
K-theory on which the twistings are non-trivial.

\subsection{Stringy product on twisted orbifold K-theory}
In this section we will review the definition of the twisted
orbifold K-theory. For the details on the construction see
\cite{AdemRuan, AdemRuanZhang}.

An element $\beta \in H^3_G(M; \integer)$ defines
$\tK{\beta}_G^*(M)$, the twisted $G$-equivariant K-theory of $M$.
If $\beta$ comes from $H^3(BG; \integer)$, then $\beta$ defines a
central extension of $G$ by the circle $$ 1 \to S^1 \to
\widetilde{G} \to G \to 0$$ and the twisted $G$-equivariant
K-theory $\tK{\beta}^*_G(M)$ becomes the group of
$\widetilde{G}$-equivariant vector bundles over $M$ such that
$S^1$ acts by multiplication on the fibers.

For $\alpha \in H^4_G(M;\integer)$ the stringy product of the
twisted orbifold K-theory $\tK{\alpha}_{st}^*([M/G])$ is defined
as the group
$$\tK{\alpha}_{st}^*([M/G]):= \bigoplus_{g \in G}
\tK{\alpha_g}_{G}^*(M^g) \times \{g\}$$ where $\alpha_g \in
H^3_{G}(M^g; \integer)$ is a twisting class that is the image of
$\alpha$ under the inverse transgression map $\tau_g :H^4_G(M;\integer)
\to H^3_{G}(M^g; \integer)$ (see \ref{transgression map}).

Let's denote the inclusions by
$$e_1 : M^{g,h} \to M^g \ \ \ \ \ e_2: M^{g,h} \to M^h \ \ \ \ \ \ e_3: M^{g,h} \to M^{gh}$$
and recall that $D_{g,h} \in K_G^*(M^{g,h})$ is the obstruction
bundle of definition \ref{Obstruction bundle}.

\begin{definition} \label{definition twisted stringy product}
For the almost complex orbifold $[M/G]$,  define the stringy
product  $\star$ in the twisted orbifold K-theory
$\tK{\alpha}_{st}^*([M/G])$ as follows:
\begin{eqnarray*}
 (\tK{\alpha_g}^*_G(M^g)\times\{g\}) \times
(\tK{\alpha_h}^*_G(M^h) \times \{h\}) & \stackrel{\star}{\to} &
\tK{\alpha_{gh
}}^*_G(M^{gh})\times\{gh\} \\
((F,g),(H,h)) & \mapsto & (e_{3!}(e_1^*F \otimes e_2^*H \otimes
\lambda_{-1}(D_{g,h})),gh).
\end{eqnarray*}
\end{definition}

The proof of the fact that the $\star$ operation makes
$\tK{\alpha}_{st}^*([M/G])$ into a ring can be found in
\cite{AdemRuanZhang}. The product is well defined because
$\alpha_g + \alpha_h = \alpha_{gh}$ and the associativity of
$\star$ follows from the properties of the obstruction bundles as
in section \ref{subsection obstruction bundle}.

\subsection{Inverse transgression map}

In this section we will show that the transgressed classes
$\alpha_g$ are all torsion.

Let us recall the definition of the inverse transgression map for the
orbifold $[M/G]$ when $G$ is finite not necessarily abelian.

For $g \in G$ let $C(g)$ denote the centralizer of $g$ in $G$.
Consider the action of $C(g) \times \integer$ on $M^g$ given by
$(h,m)\cdot x := x hg^m$ and the homomorphism
\begin{eqnarray*} \phi_g : C(g)
\times  \integer &  \to  & G\\
(h,m) & \mapsto & hg^m.
\end{eqnarray*} Then the inclusion $i_g : M^g \to M$ becomes a
$\phi_g$-equivariant map and it induces a homomorphism
$$i_g^* : H_G^*(M; \integer) \to H^*_{C(g) \times Z}(M^g ; \integer).$$
As the group $\integer$ acts trivially on $M^g$ then we have that
$$H^*_{C(g) \times Z}(M^g ; \integer) \cong H^*(M^g
\times_{C(g)}EC(g) \times B\integer ; \integer),$$ and as
$B\integer \simeq S^1$, via the Kunneth isomorphism we get that
$$H^*_{C(g) \times Z}(M^g ; \integer) \cong H^*_{C(g)}(M^g;
\integer) \otimes_\integer H^*(S^1 ; \integer).$$

Therefore $$i_g^* : H^p_G(M ; \integer) \to H^p_{C(g)}(M^g
;\integer) \oplus H^{p-1}_{C(g)}(M^g ;\integer)$$ and projecting
on the second summand we get the inverse transgression map:
\begin{eqnarray} \label{transgression map} \tau_g : H^p_G(M; \integer) \to H^{p-1}_{C(g)}(M^g;
\integer).\end{eqnarray}

\begin{lemma}
The transgressed classes are all torsion, i.e. for $\alpha \in
H^p_G(M; \integer)$ the inverse transgression $\alpha_g : = \tau_g(\alpha)
\in H^{p-1}_{C(g)}(M^g; \integer)$ is a torsion class.
\end{lemma}

\begin{proof}
Let $\langle g \rangle$ denote the cyclic group generated by $g$.
Then the homomorphism $\phi_g$ factors through
$$C(g) \times \integer \to C(g) \times \langle g \rangle \to G  \
\ \ \ \ \ \ \ \ (h,m) \mapsto (h, g^m) \mapsto hg^m$$ and
therefore $i_g^*$ factors through $$H^*_G(M; \integer) \to H^*(M^g
\times_{C(g)} EC(g) \times B \langle g \rangle ; \integer) \to
H^*(M^g \times_{C(g)} EC(g) \times B \integer ; \integer).$$

Now, as the cohomology of $B \langle g \rangle$ is all torsion,
the Kunneth isomorphism tells us that the torsion free part of
$H^p(M^g \times_{C(g)} EC(g) \times B \langle g \rangle ;
\integer)$ comes from $$H^p_{C(g)}(M^g ; \integer)
\otimes_\integer H^0(B \langle g \rangle ; \integer).$$ As the
inverse transgression map does not factor through $H^p_{C(g)}(M^g ;
\integer) \otimes_\integer H^0(B \langle g \rangle ; \integer)$,
then the transgressed classes must be torsion.

\end{proof}

\begin{example}
Let's calculate the inverse transgression map for the group $G=(
\integer/2)^3$; here we follow the analysis of lemma 5.2 in
\cite{AdemRuanZhang}.

 Let's suppose first that the coefficients
are in the field $\field_2$ of two elements. Then we have that
$$H^*(G; \field_2) \cong \field_2 [x_1, x_2, x_3]$$ is the
polynomial algebra on three generators. Take  $g = (a_1, a_2, a_3)
\in G$ and consider the maps
$$G \times \integer \to G \times \integer/2
\to G \ \ \ \ \ \
(h,m) \mapsto (h, mg) \mapsto h + mg.$$
\end{example}

In cohomology we get \begin{eqnarray*}& H^*(BG, \field_2)  \to
H^*(BG \times B \integer/2 ; \field_2) \to H^*(BG \times B\integer ; \field_2)& \\
& \field_2[x_1,x_2,x_3]  \to  \field_2[x_1,x_2,x_3,w] \to  \field_2[x_1,x_2,x_3,w]/(w^2) & \\
 & x_i \ \  \mapsto \ \ x_i + a_i w \ \ \mapsto \ \  x_i + a_i w &
\end{eqnarray*}
and therefore the inverse transgression map with $\field_2$ coefficients
is
\begin{eqnarray*}
& \tau_g : \field_2[x_1,x_2,x_3]   {\to}  \field_2[x_1,x_2,x_3] & \\
& x_1^{m_1} x_2^{m_2} x_3^{m_3}  \mapsto \left(
a_1m_1x_1^{m_1-1}x_2^{m_2} x_3^{m_3} + a_2m_2x_1^{m_1} x_2^{m_2-1}
x_3^{m_3} + a_3m_3x_1^{m_1} x_2^{m_2} x_3^{m_3-1}\right). &
\end{eqnarray*}

For abelian 2-groups, the map $H^k(BG ; \integer) \to H^k (BG ;
\field_2) $ is a monomorphism, and as $H^k(BG ; \integer) \cong
H^k(BG ; \integer/4)$, then the long exact sequence induced by the
extension $0 \to \integer/2 \to \integer/4 \to \integer/2 \to 0$
implies that $$H^k(BG; \integer) \cong \mbox{kernel} \left( Sq^1 :
H^k(BG ; \field_2) \to H^{k+1}(BG; \field_2) \right);$$ here $Sq^1$ is the
Steenrod square operation that in this case is equivalent to the
Bockstein map.

If we want to find elements in $H^4(BG ; \integer)$ that induce
non trivial twistings via the inverse transgression map, one needs to find
a homogenous polynomial $p(x_1, x_2, x_3) \in
\field_2[x_1,x_2,x_3]$ of degree 4, such that $Sq^1(p) =0$ and
that $\tau_g(p) \neq 0$ for some $g \in G$. In \cite[Lemma
5.2]{AdemRuanZhang} it is shown that the element
\begin{eqnarray} \label{definition of alpha}
\alpha=Sq^1(x_1x_2x_3) = x_1^2x_2x_3 + x_1x_2^2x_3 +x_1x_2x_3^2
\end{eqnarray}
satisfies these properties. Its inverse transgression in $\tau_g: H^4(BG
;\integer) \to H^3(BG, \integer)$ for $g=(a_1,a_2,a_3)$ is equal
to
\begin{eqnarray} \label{transgression alpha}
\alpha_g = \tau_g(\alpha) = a_1(x_2^2 x_3 + x_2x^2_3) +a_2(x_3^2
x_1 + x_3x^2_1)+a_3(x_1^2 x_2 + x_1x^2_2).
\end{eqnarray}
 Let us calculate the
double inverse transgression of $\alpha$.

\begin{lemma} \label{lemma double trasngression alpha}
Let $g= (a_1, a_2, a_3)$ and $h = (b_1, b_2, b_3)$ be elements in
$G= ( \integer/2)^3$. The double inverse transgression of $\alpha$ is
equal to:
\begin{eqnarray*}
\tau_h \tau_g \alpha & = & (a_2 b_3 + a_3 b_2) x_1^2 +(a_3 b_1 +
a_1 b_3) x_2^2 + (a_1 b_2 + a_2 b_1) x_3^2\\
& = & \left[(a_1,a_2,a_3) \times (b_1,b_2,b_3) \right] \cdot
(x_1^2,x_2^2, x_3^2).
\end{eqnarray*}
and therefore  $\tau_h \tau_g \alpha \neq 0$ if and only if $g
\neq h$ and $g \neq 0 \neq h$.
\end{lemma}

\begin{proof}
The formula follows from the fact that the inverse transgression  $\tau_h
(x_i^2x_j)$ is $b_jx_i^2$. Now, the cross product $(a_1,a_2,a_3)
\times (b_1,b_2,b_3)$ is equal to zero only when either $g=0$ or
$h=0$, or $g=h$.
\end{proof}

\subsection{Decomposition formula for twisted orbifold K-theory}

This section is devoted to show a decomposition formula for the
twisted orbifold K-theory that is suited to study the stringy
product; this formula is a simple generalization of \cite[Thm.
7.4]{AdemRuan}. On what follows we will elaborate on the results
of sections 5 and 6 of \cite{AdemRuan}.

Let's assume that $G$ is abelian and that $\beta : G \times G \to
S^1$ is a cocycle representing the twisting class which correspond
to the group extension $1 \to S^1 \to \widetilde{G}_\beta \to G
\to 0$.
 Therefore $\widetilde{G}_\beta$  is isomorphic to the group $G \times
 S^1$ with the product given by $(g_1, c_1) (g_2,c_2) = (g_1g_2,
 \beta(g_1,g_2)c_1c_2).$

 We will start by recalling some fact about projective representations that can be
 found on \cite{AdemRuan}. Recall that the group $R_\beta(G):= \tK{\beta}_G^*(pt)$ of $\beta$-twisted
 representations can be understood as the subgroup of
 $R(\widetilde{G}_\beta)$ generated by representations where $S^1$
 acts by scalar multiplication.

 To understand the group $R_\beta(G)$ we will restrict it to all
 the cyclic subgroups of $G$; let's do this choosing a generator on each cyclic subgroup.
 Then, for $g$ in $G$ and $C=\langle g \rangle$ the cyclic group it
 generates, by restriction we have a map $R_\beta(G) \to
 R_{res(\beta)}(\langle g \rangle)$ where $res(\beta)$ is the
 restriction of $\beta$ to the group $\langle g \rangle$. As $H^2(\langle g
 \rangle; S^1)=0$ we have that $res(\beta)$ is cohomologous to
 zero, and therefore $R_{res(\beta)}(\langle g \rangle) \cong
 R(\langle g \rangle)$ the untwisted representation ring of $\langle g \rangle$.

Now, these restrictions to the cyclic subgroups are endowed with
an action of the group $G$; for this we need to make use of the
inverse transgression of the cocycle $\beta$.
 The inverse transgression of $\beta$ for $h \in G$ is a 1-cocycle, and
 therefore is given by a
 homomorphism $\tau_h \beta : G \to S^1$, where $g\mapsto
 \beta(h,g)\beta(g,h)^{-1}$ (see \cite[Lemma 6.4.1]{LupercioUribeLoopGroupoid}).
  Note that these homomorphisms also have the property that $\tau_{h}(g) \cdot
 \tau_{k}(g) = \tau_{hk}(g)$.

\vspace{0.5cm}

 Having the inverse transgression map in mind, we can show now the natural action of ${G}$ on
 $R_{res(\beta)}(\langle g \rangle)$. So, given a representation $$\rho
 : \widetilde{\langle g \rangle}_\beta \to Aut(V),$$ for $(x,b) \in \widetilde{\langle g
 \rangle}_\beta$ and $z \in {G}$ we can define
 the representation $$\left[ z \cdot \rho \right] :\widetilde{\langle g \rangle}_\beta \to
 Aut(V)$$
 such that
 $$\left[ z \cdot \rho \right]( x, b) := \rho(
 (z,1)(x,b)(z,1)^{-1}).$$ Note that this representation has the property that \begin{eqnarray}
 \label{G equivariantness} \left[ z \cdot \rho \right]( x,
 b)= \tau_x(z) \rho(x,b)\end{eqnarray} and therefore as $\tau_x$ is a homomorphism one has that $[z_2 \cdot
 [z_1 \cdot \rho]] = [z_1z_2 \cdot \rho]$.
This shows that there is a natural action of $G$ on
$R_{res(\beta)}(\langle g \rangle)$.

We have that the class $res(\beta)$ is trivial in cohomology,
therefore there is a function $\gamma : G \to S^1$ such that
$\delta \gamma = \beta$. Then we can make a $res(\beta)$-twisted
representation of $\langle g \rangle$ into a representation of
$\langle g \rangle$ in the following way. A $res(\beta)$-twisted
representation of $\langle g \rangle$ is given by a function $\psi
: \langle g \rangle  \to Aut(V)$ where $\psi(hk) =
[res(\beta)(h,k)]\psi(h)\psi(k)$. But as $res(\beta)(h,k)=
\gamma(k) \gamma(hk)^{-1} \gamma(h)$  we have that
$\gamma(hk)\psi(hk) = \gamma(h)\psi(h)\gamma(k)\psi(k)$ and
therefore $\gamma(\cdot)\psi(\cdot)$ becomes a $\langle g \rangle$
representation and we get the isomorphism
$$A_{\beta,g} : R_{res(\beta)}(\langle g \rangle) \cong R(\langle g\rangle); \ \ \ \ \
\psi \mapsto \gamma\psi.$$

Now take $\zeta$ to be a $|g|$-primitive root of unity,
$\rational(\zeta)$ the cylotomic field it generates and
 $$\chi_g: R(\langle g \rangle) \to \rational(\zeta)$$ the trace map
on the element $g$, i.e. $\chi_g \rho := \mbox{Trace}(\rho(g))$.
Endow the group $R(\langle g \rangle)$ with the $G$ action coming
from the isomorphism $A_{\beta,g }$ and the cyclotomic field
$\rational(\zeta)$ with the $G$ action given by multiplication
with $\tau_g\beta$, i.e. for $h \in G$ and $p(\zeta) \in
\rational(\zeta)$ let $h\cdot p(\zeta) = \tau_g\beta(h) p(\zeta)$.
In view of equation (\ref{G equivariantness}) the maps
$A_{\beta,g}$ and $\chi_g$ become $G$-equivariant.

As the $G$ action on $R(\langle g \rangle)$ and $\rational(\zeta)$ depends on $\beta$, we will
label these groups by $R(\langle g \rangle)_\beta$ and $\rational(\zeta)_\beta $; the subindex $\beta$
keeps track on how the $G$ action is defined. Let's now show the result on which the calculation of the
stringy product is based.

\begin{lemma}\label{lemma tensor product restrictred reps}
Let $\beta_1,\beta_2 : G \times G \to S^1$ be two 2-cocycles
representing two twisting classes. Then the diagram $$ \xymatrix{
R_{res(\beta_1)}(\langle g\rangle) \times R_{res(\beta_2)}(\langle
g\rangle) \ar[rr]^\otimes \ar[d]_{A_{\beta_1,g} \times
A_{\beta_2,g}} & & R_{res(\beta_1\beta_2)}(\langle g\rangle)
\ar[d]^{A_{\beta_1\beta_2,g}}\\
R(\langle g \rangle)_{\beta_1} \times  R(\langle g \rangle)_{\beta_2} \ar[d]_{\chi_g \times \chi_g} \ar[rr]^\otimes
&& R(\langle g \rangle)_{\beta_1\beta_2} \ar[d]^{\chi_g} \\
\rational(\zeta)_{\beta_1} \times \rational(\zeta)_{\beta_2} \ar[rr]^\cdot &&\rational(\zeta)_{\beta_1\beta_2}
 }$$ is commutative and
$G$-equivariant
\end{lemma}
\begin{proof}
Let $\psi_i : \langle g \rangle \to Aut(V_i)$ be an element in
$R_{res(\beta_i)}(\langle g\rangle)$ and let $\gamma_i: G \to S^1$
be the function such that $\delta \gamma_i = \beta_i$. We have
that $A_{\beta_i,g}(\psi_i) = \gamma_i \psi_i$ and therefore
$$A_{\beta_1,g}(\psi_1) \otimes A_{\beta_2,g}(\psi_2) = \gamma_1
\psi_1 \otimes \gamma_2 \psi_2 = \gamma_1\gamma_2 ( \psi_1 \otimes
\psi_2)= A_{\beta_1\beta_2,g}(\psi_1 \otimes \psi_2)$$ where the
last equality follows from the fact that
$\delta(\gamma_1\gamma_2)=\beta_1\beta_2$.

The commutativity of the second diagram follows from the fact that
$$\mbox{Trace}(A\otimes B)=\mbox{Trace}(A)\mbox{Trace}(B).$$

Now, for $h \in G$ we have that
$$(h \cdot \rho_1) \otimes (h \cdot \rho_2) = \tau_g\beta_1(h)
\rho_1 \otimes \tau_g\beta_2(h) \rho_2= \tau_g\beta_1\beta_2(h)
\rho_1 \otimes \rho_2= h \cdot( \rho_1 \otimes \rho_2)$$ where the
middle equality follows from the fact that $\tau_g$ is a
homomorphism. This implies the equivariantness for the three horizontal lines.
\end{proof}

\vspace{0.5cm}

The previous lemma shows how the restricted twisted representations behave once the are
tensored, but we still need to show what is the appropriate decomposition for the twisted K-theory
on which we can use lemma \ref{lemma tensor product restrictred reps}.

Let $X$ be a $G$-CW complex and $\beta: G \times G \to S^1$  a 2-cocycle representing the twisting class.
Denote $X^g$ the fixed point set of $g$ and consider the following homomorphisms
\begin{eqnarray*}
\tK{\beta}_G^*(X) \to \tK{res(\beta)}_{\langle g \rangle}^*(X^g)
\stackrel{\cong}{\to} K^*(X^g) \otimes R_{res(\beta)}(\langle g
\rangle) \to  K^*(X^g) \otimes \rational(\zeta)_\beta
\end{eqnarray*}
where the first one is the restriction map, the second is the natural isomorphism given by the trivial action
of $g$ in $X^g$ and the last one is given by the composition $\chi_g \circ A_{\beta, g} $. The composition of these
homomorphisms has image in the invariants under the $G$ action, then putting all these maps together for all the cyclic
subgroups of $G$ we get the following result

\begin{theorem}[Decomposition for twisted K-theory] \label{decomposition formula}
Let $X$ be a $G$-CW complex, $G$ a finite group and $\beta : G \times G \to S^1$ a 2-cocycle. Let $T$ be a subset
of $G$ such that any cyclic group of $G$ is generated by only one element in $T$. Then we have a decomposition
$$\tK{\beta}_G^*(X) \otimes \rational \cong \prod_{g \in T} \left( K(^*X^g) \otimes \rational(\zeta_{|g|})_\beta\right)^G$$
where $\zeta_{|g|}$ is a primitive $|g|$-th root of unity.
\end{theorem}

\begin{proof}
The isomorphism follows directly from the decomposition formula of Adem and Ruan \cite[Theorem 7.4]{AdemRuan}.

Once checks first that it holds for $X=G/H$, and then one proceeds by induction on the number of $G$ cells in $X$
with the use of the the Mayer-Vietoris exact sequence.
\end{proof}

In view of lemma (\ref{lemma tensor product restrictred reps}) we have that
\begin{theorem} \label{lemma product} The following diagram is commutative
$$\xymatrix{\tK{\beta_1}_G^*(X) \times \tK{\beta_2}_G^*(X)\ar[d] \ar[r]^\otimes & \tK{\beta_1\beta_2}_G^*(X)\ar[d] \\
\left( K^*(X^g) \otimes \rational(\zeta_{|g|})_{\beta_1}\right)^G
\times \left( K^*(X^g) \otimes
\rational(\zeta_{|g|})_{\beta_2}\right)^G \ar[r] & \left(
K^*(X^g) \otimes \rational(\zeta_{|g|})_{\beta_1\beta_2}\right)^G
}$$ where the application of the bottom line takes the pair $(E_1
\otimes p_1(\zeta), E_2
\otimes p_2(\zeta) ) $  and
maps it to $(E_1 \otimes E_2) \otimes (p_1(\zeta)p_2(\zeta))$.
\end{theorem}
Therefore the stringy product in twisted orbifold K-theory can be calculated with the use of the decomposition formula
of theorem \ref{decomposition formula}, and the fact that the stringy product structure decomposes as is shown in lemma \ref{lemma product}.

\begin{cor}
In the case that the twistings $\beta_1$ and $\beta_2$ are zero we
recover the ring isomorphism of \cite[Thm 5.1]{AdemRuan}
$$K_G^*(X) \otimes \rational \cong \prod_{g \in T} \left( K^*(X^g) \otimes \rational(\zeta_{|g|})\right)^G.$$
\end{cor}

\subsection{Examples}

We finish this chapter by showing how the decomposition formula of
theorem \ref{decomposition formula} is suited for calculating the
stringy product in twisted orbifold K-theory. Let
$G=(\integer/2)^3$ and $\alpha \in H^4(G, \integer)$ as in formula
(\ref{definition of alpha}),  we will calculate
$\tK{\alpha}_{st}^*([*/G])$ and $\tK{\alpha}_{st}^*([T^6/G])$.

Let us remark that the ring $\tK{\alpha}_{st}^*([*/G])$ has already been calculated by A. N. Duman in \cite{Duman} without tensoring
with $\rational$ by finding out the twisted representations and their products. We include here
the calculation of this ring to make use of our decomposition formula and to clarify the calculation of
$\tK{\alpha}_{st}^*([T^6/G])$.

\begin{example}
 $\tK{\alpha}^*_{st}([*/G])$

As a group we have that
$$ \tK{\alpha}_{st}^*([*/G]) = \bigoplus_{g \in G} \tK{\alpha_g}_G^*(pt) =
\bigoplus_{g \in G} R_{\alpha_g}(G)$$ where $\alpha_g$ is the
inverse transgression of $\alpha$ in $H^3(G, \integer)$ defined in formula
(\ref{transgression alpha}).

By the decomposition formula we have that
$$R_{\alpha_g}(G) \otimes \rational \cong \prod_{h \in G}
\left( \rational(\zeta_{|h|})_{\alpha_g} \right)^G = \prod_{h \in
G} \left( \rational_{\alpha_g} \right)^G$$ where $T=G$ as all the
elements generate a different cyclic group and
$\rational(\zeta_{|h|})= \rational$ because $\zeta_{|h|}$ is
either $1$ or $-1$.

Recall that the action of $k$ in
$\rational(\zeta_{|h|})_{\alpha_g}$, for $k \in G$, is given by
multiplication by $\tau_h \alpha_g(k)$, the inverse transgression of
$\alpha_g$ evaluated in $k$. Therefore, by lemma \ref{lemma double
trasngression alpha} we have that
$$V_{g,h} =: \left(\rational(\zeta_{|h|})_{\alpha_g} \right)^G =
\left\{
\begin{array}{cl}
\rational & \mbox{for} \  g=h, \ \mbox{or} \ g=0 \ \mbox{or} \ h=0
\\
0 & \mbox{otherwise}
\end{array}\right.$$
Therefore we have that for $g \neq 0$
$$R_{\alpha_g}(G) \otimes \rational \cong V_{g,1} \times V_{g,g}
\cong \rational \times \rational,$$ and as $\alpha_1 = \tau_1
\alpha=0$ we have that
$$R_{\alpha_1}(G) \otimes \rational = R(G) \otimes \rational \cong
\prod_{h \in G} \rational$$ where this last isomorphism is an
isomorphism of rings. Therefore we have that the rank of
$\tK{\alpha}_{st}^*([*/G])$ is $2 \times 7 + 8 = 22$. We point out
here that the rank of this ring was already calculated in
\cite[Example 6.2]{AdemRuanZhang}.

The product structure is easy to determine following lemma
\ref{lemma product}. Let's denote by $1_{g,h} \in \prod_{h \in G}
V_{g,h}$ the element which is $1$ in the component of $h$ and $0$
on the other components. Then we have that
$$1_{g_1,h_1} \cdot 1_{g_2,h_2 } = \left\{
\begin{array}{cl}
1_{1,h_1} & \mbox{if} \ 0=g_1g_2 \ \mbox{and} \ h_1 = h_2 \ \\
1_{g_1g_2,h_1} & \mbox{if} \ h_1 = h_2= g_1g_2 \neq 0 \ \\
 0 &\mbox{otherwise}
\end{array}
\right.$$

We finish this example by noting that the ring generated by
$\bigoplus_{g \in G} V_{g,1}$ is isomorphic to the group ring
$\rational[G]$. Therefore we have that both the group ring
$\rational[G]$ and the representation ring $R(G) \otimes
\rational$ are subrings of $\tK{\alpha}_{st}^*([*/G])$; the last
one because $R_{\alpha_1}(G)=R(G)$.

\end{example}

\begin{example} $\tK{\alpha}_{st}^*([T^6/G])$\\
Let's consider the torus $T^2= \complex/(\integer \oplus
\integer)$ together with the $\integer/2$ action given by
multiplication by $-1$. The action fixes four points $\{0,
\frac{1}{2}, \frac{1}{2}i, \frac{1}{2} + \frac{1}{2}i \}$ and we
will denote this set by $4*$. The group $\integer/2$ acts on
$K^*(T^2)$ by the trivial action on $K^0(T^2)$ and by the sign
action on $K^1(T^2)$.

Then we can consider the group $G=(\integer/2)^3$ acting
coordinate-wise on $M=T^6= T^2 \times T^2 \times T^2$; that is,
for $g=(a_1,a_2,a_3) \in G$ and $z=(z_1,z_2,z_3)\in T^6$, the
action is $g\cdot z := ((-1)^{a_1}z_1, (-1)^{a_2}z_2,
(-1)^{a_3}z_3)$.

On what follows we will explicitly calculate the rank of
$\tK{\alpha}_{st}^*([T^6/G])$ and we will show how the ring
structure behaves for some chosen elements. We will not calculate
how the ring structure behaves for all the cases because the rank
of the twisted orbifold K-theory is too big.

Let's start by calculating the rank of
$\tK{\alpha}_{st}^*([T^6/G])$.

For $g \in G$, by theorem \ref{decomposition formula} we have that
$$\tK{\alpha_g}_G^*(M^g) \otimes \rational \cong \prod_{h \in G}
(K^*(M^{g,h}) \otimes \rational_{h,\alpha_g})^G$$ where $G$ acts
on $\rational_{h,\alpha_g}=\rational$ multiplying by the double
inverse transgression $\tau_h \tau_g \alpha=\tau_h \alpha_g$. Therefore
for the elements $g \in \{(0,0,0), (1,0,0), (1,1,0), (1,1,1)\}$ we
will elaborate some tables where we will calculate the double
inverse transgressions $\tau_h \tau_g \alpha$, the fixed points $M^{g,h}$,
and the invariant set $(K^*(M^{g,h}) \otimes
\rational_{h,\alpha_g})^G$ together with its  rank.

\begin{center}
\begin{tabular}{|c|cccc|}
    \hline
    \multicolumn{5}{|c|}{$\tK{\alpha_g}_G^*(M^g) \otimes \rational$  \ \ for \ $g=(1,0,0)$} \\
    \hline
$h$ & $ \tau_h \tau_g \alpha $   &  $ M^{g,h}$  & $ (K(M^{g,h}) \otimes \rational_{h,\alpha_g})^G$  & rank  \\
    \hline
$(0,0,0)$   & $0$  & $4* \times T^2 \times T^2$ & $ K^0(4*) \otimes K^0(T^2) \otimes K^0(T^2)$  & $16$ \\
$(1,0,0)$  & $0$  & $4* \times T^2 \times T^2$ &  $ K^0(4*) \otimes K^0(T^2) \otimes K^0(T^2)$ & $16$  \\
$(0,1,0)$   & $x_3^2$ & $ 4* \times 4* \times T^2$&  $K^0(4*) \otimes K^0(4*) \otimes K^1(T^2)$ & $32$  \\
$(0,0,1)$  & $x_2^2$ &  $4* \times T^2 \times 4*$ &  $K^0(4*) \otimes K^1(T^2) \otimes K^0(4*)$ & $32$ \\
$(1,1,0)$   & $x_3^2$ & $4* \times 4* \times T^2$ &  $K^0(4*) \otimes K^0(4*) \otimes K^1(T^2)$ & $32$  \\
$(0,1,1)$   & $x_2^2+x_3^2$ & $4* \times 4* \times 4*$ & $0 $ & $0$  \\
$(1,0,1)$ & $x_2^2$ & $4* \times T^2 \times  4*$ & $K^0(4*) \otimes K^1(T^2) \times K^0(4*)$ & $32$ \\
$(1,1,1)$ & $x_2^2+ x_3^2$ & $4* \times 4* \times 4*$ & $ 0$  &  \\
\hline & & & Total rank & $ 160$ \\
\hline
\end{tabular}
\end{center}

\begin{center}
\begin{tabular}{|c|cccc|}
    \hline
    \multicolumn{5}{|c|}{$\tK{\alpha_g}_G^*(M^g) \otimes \rational$  \ \ for \ $g=(1,1,0)$} \\
    \hline
$h$ & $ \tau_h \tau_g \alpha $   &  $ M^{g,h}$  & $ (K(M^{g,h}) \otimes \rational_{h,\alpha_g})^G$  & rank  \\
    \hline
$(0,0,0)$   & $0$  & $4* \times 4* \times T^2$ & $ K^0(4*) \otimes K^0(4*) \otimes K^0(T^2)$  & $32$ \\
$(1,0,0)$  & $x_3^2$  & $4* \times 4* \times T^2$ &  $K^0(4*) \otimes K^0(4*) \otimes K^1(T^2)$ & $32$  \\
$(0,1,0)$   & $x_3^2$ & $ 4* \times 4* \times T^2$&  $K^0(4*) \otimes K^0(4*) \otimes K^1(T^2)$ & $32$  \\
$(0,0,1)$  & $x_1^2+x_2^2$ &  $4* \times 4* \times 4*$ &  $0$ & $0$ \\
$(1,1,0)$   & $0$ & $4* \times 4* \times T^2$ &  $K^0(4*) \otimes K^0(4*) \otimes K^0(T^2)$ & $32$  \\
$(0,1,1)$   & $x_1^2+x_2^2+x_3^2$ & $4* \times 4* \times 4*$ & $0 $ & $0$  \\
$(1,0,1)$ & $x_1^2+x_2^2+x_3^2$ & $4* \times 4* \times  4*$ & $0$ & $0$ \\
$(1,1,1)$ & $x_1^2+x_2^2$ & $4* \times 4* \times 4*$ & $ 0$  & 0 \\
\hline & & & Total rank & $ 128$ \\
\hline
\end{tabular}
\end{center}

\begin{center}
\begin{tabular}{|c|cccc|}
    \hline
    \multicolumn{5}{|c|}{$\tK{\alpha_g}_G^*(M^g) \otimes \rational$  \ \ for \ $g=(1,1,1)$} \\
    \hline
$h$ & $ \tau_h \tau_g \alpha $   &  $ M^{g,h}$  & $ (K(M^{g,h}) \otimes \rational_{h,\alpha_g})^G$  & rank  \\
    \hline
$(0,0,0)$   & $0$  & $4* \times 4* \times 4*$ & $ K^0(4*) \otimes K^0(4*) \otimes K^0(4*)$  & $64$ \\
$(1,0,0)$  & $x_2^2+x_3^2$  & $4* \times 4* \times 4*$ &  $0$ & $0$  \\
$(0,1,0)$   & $x_1^2+x_3^2$ & $ 4* \times 4* \times 4*$&  $0$ & $0$  \\
$(0,0,1)$  & $x_1^2+x_2^2$ &  $4* \times 4* \times 4*$ &  $0$ & $0$ \\
$(1,1,0)$   & $x_1^2+x_2^2$ & $4* \times 4* \times 4*$ &  $0$ & $0$  \\
$(0,1,1)$   & $x_2^2+x_3^2$ & $4* \times 4* \times 4*$ & $0 $ & $0$  \\
$(1,0,1)$ & $x_1^2+x_3^2$ & $4* \times 4* \times  4*$ & $0$ & $0$ \\
$(1,1,1)$ & $0$ & $4* \times 4* \times 4*$ & $ K^0(4*) \otimes K^0(4*) \otimes K^0(4*)$  & 64 \\
\hline & & & Total rank & $ 128$ \\
\hline
\end{tabular}
\end{center}
\noindent For $g=(0,0,0)$ the inverse transgression $\tau_g\alpha=0$ and
therefore we just need to calculate the invariants for each $h$
\begin{center}
\begin{tabular}{|c|ccc|}
    \hline
    \multicolumn{4}{|c|}{$K_G^*(M^g) \otimes \rational$  \ \ for \ $g=(0,0,0)$} \\
    \hline
$h$ &   $ M^{g,h}$  & $ (K(M^{g,h}) \otimes \rational)^G$  & rank  \\
    \hline
$(0,0,0)$    & $T^2 \times T^2 \times T^2$ & $ K^0(T^2) \otimes K^0(T^2) \otimes K^0(T^2)$  & $8$ \\
$(1,0,0)$    & $4* \times T^2 \times T^2$ &  $ K^0(4*) \otimes K^0(T^2) \otimes K^0(T^2)$ & $16$  \\
$(0,1,0)$    & $ T^2 \times 4* \times T^2$&  $K^0(T^2) \otimes K^0(4*) \otimes K^0(T^2)$ & $16$  \\
$(0,0,1)$  &  $T^2 \times T^2 \times 4*$ &  $K^0(T^2) \otimes K^0(T^2) \otimes K^0(4*)$ & $16$ \\
$(1,1,0)$    & $4* \times 4* \times T^2$ &  $K^0(4*) \otimes K^0(4*) \otimes K^0(T^2)$ & $32$  \\
$(0,1,1)$    & $T^2 \times 4* \times 4*$ & $K^0(T^2) \otimes K^0(4*) \otimes K^0(4*)$ & $32$  \\
$(1,0,1)$  & $4* \times T^2 \times  4*$ & $K^0(4*) \otimes K^0(T^2) \times K^0(4*)$ & $32$ \\
$(1,1,1)$  & $4* \times 4* \times 4*$ & $K^0(4*) \otimes K^1(4*) \otimes K^0(4*)$  & $64$ \\
\hline & & Total rank & $ 216$ \\
\hline
\end{tabular}
\end{center}

As the groups $\tK{\alpha_g}_G^*(M^g)$ are isomorphic for
$g\in\{(1,0,0), (0,1,0), (0,0,1)\}$ and for $g\in \{(1,1,0),
(1,0,1), (0,1,1)\}$, then we can conclude that
$$\mbox{rank} \ \tK{\alpha}_{st}^*([T^6/G]) = 216 + 3 \times 160 + 3
\times 128 + 128 = 1208.$$

 \vspace{0.5cm}

The stringy product now is easy to calculate. We just need to find
the osbtruction bundle and the pushforward maps.

First note that all the obstruction bundles are trivial. This is
because the eigenvalues of the action of any element in $G$ is $1$
or $-1$, and therefore the $a_j + b_j$ of definition
\ref{Obstruction bundle} are always smaller or equal than $1$.

The pushforwards of the stringy product can be calculated by
understanding the $\integer/2$-equivariant pushforward of the
inclusion $4* \to T^2$. Let's denote by $\eta \in
K^0_{\integer/2}(T^2)$ the trivial 1-dimensional bundle whose
action of $\integer/2$ is given by multiplication by $-1$ composed
with the geometrical action, and $H \in K^0_{\integer/2}(T^2)$ be
the pullback of the Hopf bundle under the degree 1 map $T^2 \to
S^2$ together with the pullback action of $\integer/2$. If we
consider only one point $i :* \to T^2$ the pushforward map is
$$ i_! : R(\integer/2) \to K_{\integer/2}^0(T^2),  \ \ \ \
 i_!(1)= 1-H \otimes \eta$$
 and therefore $i^*i_!(1) = 1- \xi$ where $\xi$ is the sign
 representation.

For $E \in \tK{\alpha_g}_G^*(M^g)$ and $F \in
\tK{\alpha_k}_G^*(M^k)$, by definition \ref{definition twisted
stringy product}, their stringy product is
$$E \star F :=e_{3!}(e_1^*E \otimes e_2^*F ) \in \tK{\alpha_{gk}}_G^*(M^{gk}).$$
By the decomposition formula we have that $E \mapsto \prod_h E_h$
and $F \mapsto \prod_h F_h$ and by theorem \ref{lemma product} we
have that $(e_1^*E \otimes e_2^*F)_h \cong e_1^*E_h \otimes
e_2^*F_h$. The only part where one needs to be careful is with the
behavior of the pushforward map
$$e_{3!} : \tK{\alpha_{gk}}_G^*(M^{g,k}) \to \tK{\alpha_{gk}}_G^*(M^{gk})$$
under the decomposition formula.

As there would be too many cases to consider, we will only show how to calculate the stringy
 product for $g=(1,0,0)$, $k=(1,1,0)$ and
$gk=(0,1,0)$. For the other cases one can imitate the following procedure.

First let's calculate $\tK{\alpha_{gk}}_G^*(M^{g,k})$ in order to
know which components are non zero in its decomposition:

\begin{center}
\begin{tabular}{|c|ccc|}
    \hline
    \multicolumn{4}{|c|}{$\tK{\alpha_{gk}}_G(M^{g,k}) \otimes \rational$  \ \ for \ $g=(1,0,0)$ \ and $k=(1,1,0)$} \\
    \hline
$h$ & $ \tau_h \tau_{gk} \alpha $   &  $ M^{g,k,h}$  & $ (K(M^{g,k,h}) \otimes \rational_{h,\alpha_{gk}})^G$    \\
    \hline
$(0,0,0)$   & $0$  & $4* \times 4* \times T^2$ & $ K^0(4*) \otimes K^0(4*) \otimes K^0(T^2)$   \\
$(1,0,0)$  & $x_3^2$  & $4* \times 4* \times T^2$ &  $K^0(4*) \otimes K^0(4*) \otimes K^1(T^2)$ \\
$(0,1,0)$   & $0$ & $ 4*\times 4* \times T^2$&  $K^0(4*) \otimes K^0(4*) \otimes K^0(T^2)$ \\
$(0,0,1)$  & $x_1^2$ &  $4* \times 4* \times 4*$ &  $0$ \\
$(1,1,0)$   & $x_3^2$ & $4* \times 4* \times T^2$ &  $K^0(4*) \otimes K^0(4*) \otimes K^1(T^2)$   \\
$(0,1,1)$   & $x_1^2$ & $4* \times 4* \times 4*$ & $0 $   \\
$(1,0,1)$ & $x_1^2+x_3^2$ & $4* \times 4* \times  4*$ & $0$  \\
$(1,1,1)$ & $x_1^2+x_3^2$ & $4* \times 4* \times 4*$ & $ 0$   \\
\hline
\end{tabular}
\end{center}
Now let's see how the pushforward
$$\xymatrix{
\tK{\alpha_{gk}}_G^*(M^{g,k}) \ar[rr]^{e_{3!}} &&
\tK{\alpha_{gk}}_G^*(M^{gk}) }$$ behaves on each of the nonzero
components (on the left we have $(K^*(M^{g,k,h}) \otimes
\rational_{h,\alpha_{gk}})^G$ and on the right $(K^*(M^{gk,h})
\otimes \rational_{h,\alpha_{gk}})^G$).

\vspace{0.3cm}

\noindent For $h=(0,0,0)$
$$\xymatrix{
K^0(4*) \otimes K^0(4*) \otimes K^0(T^2) \ar[rr]^{(1-H) \otimes 1\otimes 1}&& K^0(T^2) \otimes K^0(4*) \otimes K^0(T^2)\\
}$$
where the map is given by multiplying with the element $(1-H) \otimes 1\otimes 1$ (the Thom class of the inclusion).

\vspace{0.3cm}

\noindent For $h=(1,0,0)$
$$\xymatrix{
K^0(4*) \otimes K^0(4*) \otimes K^0(T^2) \ar[rr]^{2 \otimes 1 \otimes 1}&& K^0(4*) \otimes K^0(4*) \otimes K^0(T^2)\\
}$$
because $(1-H\otimes \eta) \otimes 1 \otimes 1 $ restricted to $4* \times 4* \times T^2 $ is $(1-\xi)\otimes 1 \otimes 1$, and its character
on $h=(1,0,0)$ is $(2 \otimes 1 \otimes 1)$.

\vspace{0.3cm}

\noindent For $h=(0,1,0)$
$$\xymatrix{
K^0(4*) \otimes K^0(4*) \otimes K^0(T^2) \ar[rr]^{(1-H) \otimes 1\otimes 1}&& K^0(T^2) \otimes K^0(4*) \otimes K^0(T^2)\\
}$$
because the character of $(1-H \otimes \eta) \otimes 1 \otimes 1$ in $h=
(0,1,0)$ is $1-H \otimes 1\otimes 1$.

\vspace{0.3cm}

\noindent For $h=(1,1,0)$
$$\xymatrix{
K^0(4*) \otimes K^0(4*) \otimes K^0(T^2) \ar[rr]^{2 \otimes 1 \otimes 1}&& K^0(4*) \otimes K^0(4*) \otimes K^0(T^2)\\
}$$
because the character of $(1 - \xi) \otimes 1 \otimes 1$ in $h=(1,1,0)$ is $2 \otimes 1 \otimes 1$.

\vspace{0.3cm}

The decomposition for the stringy product $E \star F$ can be now easily calculated. Write $E_h:=(E^1_h, E^2_h,E^3_h)$
 and $F_h=(F^1_h, F^2_h,F^3_h)$
in terms of Kunneth decomposition as in the tables above.  For example, for  $h=(1,1,0)$ we have that
 $ (E^1_h, E^2_h,E^3_h) \in K^0(4*) \otimes K^0(4*) \otimes K^1(T^2)$. If we call the inclusion $j: 4* \to T^2$, then we have that
 \begin{eqnarray*}
 (E \star F)_h = \left\{
\begin{array}{ll}
 \left((1-H) \otimes E^1_h \otimes F^1_h\right) \otimes \left( j^*E^2_h \otimes F^2_h \right) \otimes \left(E^3_h \otimes F^3_h \right) &
 \mbox{for} \ h=(0,0,0)\\
 \left( 2 E^1_h \otimes F^1_h\right) \otimes \left( j^*E^2_h \otimes F^2_h \right) \otimes \left(E^3_h \otimes F^3_h \right) &
 \mbox{for} \ h=(1,0,0)\\
 \left((1-H) \otimes E^1_h \otimes F^1_h\right) \otimes \left( E^2_h \otimes F^2_h \right) \otimes \left(E^3_h \otimes F^3_h \right) &
 \mbox{for} \ h=(0,1,0)\\
 \left(2 E^1_h \otimes F^1_h\right) \otimes \left( E^2_h \otimes F^2_h \right) \otimes \left(E^3_h \otimes F^3_h \right) &
 \mbox{for} \ h=(1,1,0)\\
 0 & \mbox{otherwise}
\end{array}
 \right.
 \end{eqnarray*}

\end{example}

\bibliographystyle{alpha}
\bibliography{StringyK-theory}

\begin{thebibliography}{JKK07}

\bibitem[AR03]{AdemRuan}
A.~Adem and Y.~Ruan.
\newblock Twisted orbifold {$K$}-theory.
\newblock {\em Comm. Math. Phys.}, 237(3):533--556, 2003.

\bibitem[ARZ]{AdemRuanZhang}
A.~Adem, Y.~Ruan, and B.~Zhang.
\newblock A stringy product on twisted orbifold k-theory.
\newblock arxiv:math.AT/0605534.

\bibitem[AS68]{AtiyahSinger3}
M.~F. Atiyah and I.~M. Singer.
\newblock The index of elliptic operators. {III}.
\newblock {\em Ann. of Math. (2)}, 87:546--604, 1968.

\bibitem[AS69]{AtiyahSegal_completion}
M.~F. Atiyah and G.~B. Segal.
\newblock Equivariant {$K$}-theory and completion.
\newblock {\em J. Differential Geometry}, 3:1--18, 1969.

\bibitem[CH06]{ChenHu}
B.~Chen and S.~Hu.
\newblock A de{R}ham model for {C}hen-{R}uan cohomology ring of abelian
  orbifolds.
\newblock {\em Math. Ann.}, 336(1):51--71, 2006.

\bibitem[CR04]{ChenRuan}
W.~Chen and Y.~Ruan.
\newblock A new cohomology theory of orbifold.
\newblock {\em Comm. Math. Phys.}, 248(1):1--31, 2004.

\bibitem[Dum]{Duman}
Ali~Nabi Duman.
\newblock An example of a twisted fusion algebra.
\newblock Preprint.

\bibitem[JKK07]{JarvisKaufmannKimura}
Tyler~J. Jarvis, Ralph Kaufmann, and Takashi Kimura.
\newblock Stringy {$K$}-theory and the {C}hern character.
\newblock {\em Invent. Math.}, 168(1):23--81, 2007.

\bibitem[LO01]{LueckOliver}
W.~L{\"u}ck and B.~Oliver.
\newblock Chern characters for the equivariant {$K$}-theory of proper
  {$G$}-{CW}-complexes.
\newblock In {\em Cohomological methods in homotopy theory (Bellaterra, 1998)},
  volume 196 of {\em Progr. Math.}, pages 217--247. Birkh\"auser, Basel, 2001.

\bibitem[LU02]{LupercioUribeLoopGroupoid}
E.~Lupercio and B.~Uribe.
\newblock Loop groupoids, gerbes, and twisted sectors on orbifolds.
\newblock In {\em Orbifolds in mathematics and physics (Madison, WI, 2001)},
  volume 310 of {\em Contemp. Math.}, pages 163--184. Amer. Math. Soc.,
  Providence, RI, 2002.

\bibitem[Moe02]{Moerdijk2002}
I.~Moerdijk.
\newblock Orbifolds as groupoids: an introduction.
\newblock In {\em Orbifolds in mathematics and physics (Madison, WI, 2001)},
  volume 310 of {\em Contemp. Math.}, pages 205--222. Amer. Math. Soc.,
  Providence, RI, 2002.

\bibitem[Qui71]{Quillen}
D.~Quillen.
\newblock Elementary proofs of some results of cobordism theory using
  {S}teenrod operations.
\newblock {\em Advances in Math.}, 7:29--56 (1971), 1971.

\bibitem[Seg68]{Segal_K-theory}
G.~Segal.
\newblock Equivariant {$K$}-theory.
\newblock {\em Inst. Hautes \'Etudes Sci. Publ. Math.}, (34):129--151, 1968.

\end{thebibliography}
\end{document}